\documentclass[11pt]{article}
\usepackage{amsfonts,amsmath,amsthm,latexsym,color,epsfig}
\usepackage{here}
\usepackage[T1]{fontenc}   
\usepackage{cleveref}

\theoremstyle{plain}

\newtheorem{theorem}{Theorem}
      \newtheorem{lemma}[theorem]{Lemma}
      \newtheorem{claim}[theorem]{Claim}

\theoremstyle{definition}

\newtheorem{theoremA}{Theorem}

\input{epsf}

\def\F{\mathcal{F}}

\def\RR{\mathbb{R}}
\def\qed{\ifvmode\mbox{ }\else\unskip\fi\hskip 1em plus 10fill$\Box$}

\long\def\ignore#1{}

\def\rc{\advance\leftskip by 0pt plus 40em\rightskip=\leftskip
  \parfillskip=0pt \spaceskip=.3333em \xspaceskip=.5em
  \pretolerance=9999 \tolerance=9999 \hyphenpenalty=9999
  \exhyphenpenalty=9999}

\normalsize
\newcommand{\card}[1]{\lvert #1\rvert}

\def\bd{\mbox{\rm bd}}

\long\def\comment#1{}

\begin{document}

\title{A graded version of the separator theorem for planar graphs }
\author{Nabil H. Mustafa
\thanks{Universit\'e Sorbonne Paris Nord, Paris, France. Email: \tt{mustafa@lipn.univ-paris13.fr.}}
\and
J\'anos Pach
\thanks{R\'enyi Institute, Budapest, Hungary.
    Email: {\tt pach@cims.nyu.edu}. Supported by ERC Advanced Grant ``GeoScape.''}}

\date{}

\maketitle

\begin{abstract}
We prove that every planar graph on $n$ vertices has a vertex ordering $v_1,v_2,\ldots,v_n$ in which each prefix $\{v_1,\ldots,v_i\}$ is joined to at most $C\sqrt{i}$ elements of $\{v_{i+1},\ldots,v_n\}$,
where $C$ is an absolute constant. Consequently, every prefix can be separated
from its complement by deleting $O\left(\sqrt{i}\right)$ vertices lying entirely in its
complement.
\end{abstract}

\section{Introduction}

Given a graph $G$ with vertex set $V(G)$ and a subset $X\subseteq V(G)$, define the \emph{boundary} of $X$ in $G$ as
\[
  \bd_G X
  =\big\{v\in V(G)\setminus X \colon uv\in E(G)\text{ for some }u\in X\big\}.
\]
A \emph{prefix} of an ordering $v_1,\ldots,v_n$ is a set $P_i=\{v_1,\ldots,v_i\}$ for some $i \in [n]$; we also set $P_0:=\emptyset$.

Let $G[X]$ denote the subgraph of $G$ induced by $X$.

Our main result is the following.

\begin{theorem}\label{thm:main}
There is an absolute constant $C$ such that every planar graph $G$ on $n$
vertices has a vertex ordering $v_1,\ldots,v_n$ satisfying
\[
  \card{\bd_G \{v_1,\ldots,v_i\}}\le C\sqrt{i}\qquad \mbox{ for every } i \in [n].
\]
In particular, for every $i \in [n]$, there is a set
$S_i\subseteq\{v_{i+1},\ldots,v_n\}$ of size at most $C\sqrt{i}$ such that
there is no edge of $G$ between $P_i=\{v_1,\ldots,v_i\}$ and $V(G)\setminus(P_i\cup S_i)$.

Furthermore, this is asymptotically optimal in the following sense: for every $n$ there exists a planar graph $G$ on $n$ vertices such that for every ordering of $V(G)$ and every $1 \leq i \leq n/2$, $\card{ \bd_G \{v_1, \ldots, v_i\} } = \Omega \left( \sqrt{i} \right)$.
\end{theorem}

We have long sought, without success, to derive this theorem from one of the many ingenious geometric and graph-theoretic proofs of the Lipton--Tarjan planar graph separator
theorem~\cite{LT79}; see, e.g., \cite{AST90, AST94, HP11, MTTV97, U51}. Finally, we realized that \Cref{thm:main} can be proved using the planar graph separator theorem as a black box.

A standard version of this theorem can be stated as follows.
\begin{theoremA} \label{thm:LTsep}
There is an absolute constant $c\ge1$ such that every
planar graph $G$ on $n\ge1$ vertices admits a partition  $V(G)=A\mathbin{\dot\cup}Z\mathbin{\dot\cup}B$  satisfying
\begin{equation}\label{eq:separator}
  \card A,\card B\le\frac{2n}{3},
  \qquad
  \card Z\le c\sqrt n,
\end{equation}
and with no edge between $A$ and $B$.
\end{theoremA}
Fix such a constant $c$; it will be used throughout our proof.
\smallskip

\noindent  \textbf{Remark 1.} \Cref{thm:main} is a graded   discrete analogue
	of the isoperimetric property of the Euclidean plane.
One can view  the nested
prefixes $P_1\subseteq P_2\subseteq \cdots \subseteq P_n$ as a ball grown
vertex by vertex according to the ordering, and view $\bd_G P_i$ as the
perimeter of $P_i$. In $\RR^2$, for a ball of radius $r$, the
perimeter is $\Theta\left( \sqrt{\mathrm{area}} \right)$. By the classical isoperimetric
inequality, balls minimize perimeter among all regions of a given area.
\Cref{thm:main} provides the matching discrete graded statement: every planar
graph admits an ordering in which every prefix $P_i$ has perimeter
$O(\sqrt{|P_i|})$.

\noindent \textbf{Remark 2.} The single prefix $P_{\lfloor n/2\rfloor}$ given by \Cref{thm:main}
 yields the planar separator theorem, with separator
$\bd_G P_{ \lfloor n/2 \rfloor}$ of size $O\left(\sqrt n\right)$.

\section{Two technical lemmas}

Using the separator theorem, it is easy to prove the following weaker version of Theorem~\ref{thm:main}.

\begin{lemma}\label{lem:uniform}
There is an absolute constant $D$ such that every planar graph $G$ on $n$
vertices has a vertex ordering $v_1,\ldots,v_n$ satisfying
\[
  \card{\bd_G \{v_1,\ldots,v_i\}}\le D\sqrt{n}\qquad \mbox{for every } i \in [n].
\]
Here we can set $D=6c$, where $c$ is the constant from Theorem~\ref{thm:LTsep}.
\end{lemma}

\begin{proof}
The proof is by induction on $n$. The case $n=0$ being trivial, assume that $n>0$ and that the lemma has already been proved for all graphs with fewer than $n$ vertices.

Let $A$, $B$ and $Z$ be
the three sets obtained by applying Theorem~\ref{thm:LTsep} to $G$. Construct an ordering of $A$ by induction on $G[A]$, then append to it an ordering of $B$ constructed
by induction on $G[B]$, and finally append the  elements of $Z$ in any order.

For a prefix $P\subseteq A$, its boundary in $G$ is contained in
$Z\cup\bd_{G[A]}P$. For a prefix of the form $P=A\cup Q$ with $Q\subseteq B$ being a prefix of $B$, the boundary of $P$ in $G$ is contained in $Z\cup\bd_{G[B]}Q$. For a prefix extending into $Z$, the boundary is contained in the unlisted part of $Z$. Thus, applying the inductive hypothesis, the boundary in $G$ of every prefix with respect to the above ordering is of size at most
\[
|Z| + D \max \left\{ \sqrt{|A|}, \sqrt{|B|} \right\} \leq c\sqrt n+D\sqrt{2n/3}.
\]
The R.H.S. above is at most $D\sqrt n$ as long as $c \leq D\left(1-\sqrt{2/3}\right)$. In particular, one can set $D = 6c$, proving the inductive step.
\end{proof}

We need a second, independent, consequence of the separator theorem. 

\begin{lemma}\label{lem:shrink}
There exist absolute constants $K,M>0$ such that the following is true. For a planar graph $G$ and a set $U\subset V(G)$ with $|U|>M$ and $|\bd_G U|\le K\sqrt{|U|}$,
there exists  a subset $W\subset U$ such that
\begin{equation}\label{eq:shrink}
  \frac{3}{10}|U|\le|W|\le\frac{3}{4}|U|, \qquad \text{and} \qquad |\bd_G W|\le K\sqrt{|W|}.
\end{equation}
Here we can set $K=48c$ and $M=\left\lceil(1020c)^2\right\rceil$, where $c$ is the constant from \Cref{thm:LTsep}.
\end{lemma}

\begin{proof}
Let $m = |U|$, $B=\bd_G U$ and $b=|B|$. Applying \Cref{thm:LTsep} to the planar graph $H=G[U\cup B]$, we obtain a partition
$V(H)=L\mathbin{\dot\cup}Z\mathbin{\dot\cup}R.$
Swap $L$ and $R$, if necessary, so that
$|B\cap L|\le b/2$, and set $W=L\cap U$.

Since $m>M=\left\lceil(1020c)^2\right\rceil$, we have $\sqrt m>1020c\ge K$, so $b\le K\sqrt m\le m$. Thus
\begin{equation}\label{eq:separator-size}
  \card Z\le c\sqrt{m+b}\le2c\sqrt m.
\end{equation}
Every neighbour of every vertex in $W$ belongs to $B\cup U$. As there is
no edge between $L$ and $R$, every such neighbour outside $W$ lies in
$(B\cap L)\cup Z$. It follows that
\begin{equation}\label{eq:child-boundary}
  \card{\bd_G W}
  \le\frac b2+\card Z
  \le(K/2+2c)\sqrt m
  =26c\sqrt m.
\end{equation}

By \eqref{eq:separator-size}, and noting that $m \geq M$ is equivalent to $m \geq 1020c \sqrt{m}$,
\begin{align*}
  \card W
  &\ge m-\card Z-\card R  \geq m - 2c \sqrt{m} - \frac{2 (m+b)}{3} \geq \frac m3-\left(\frac{2K}{3}+2c\right)\sqrt m     =\frac m3-34c\sqrt m   \ge\frac{3m}{10},\\
  \intertext{and}
  \card W
  &\le\card L
   \le\frac{2(m+b)}{3}
   \le\frac{2m}{3}+32c\sqrt m
   \le\frac{3m}{4}.
\end{align*}
Finally, by \eqref{eq:child-boundary},
\[
  \card{\bd_G W} \leq 26c\sqrt m
  \le48c\sqrt{3m/10}
  \le K\sqrt{\card W}.
\]
as required.
\end{proof}

\section{Proof of Theorem~\ref{thm:main}}

Throughout this section, $D,K,M$ denote the same constants as in the above lemmas, and we set
\[  C=\max\left\{(K+D)\sqrt M,\,
                (2K+D)\sqrt{10/3}\right\}.
\]

We prove the following stronger statement, which then immediately implies \Cref{thm:main}.

\begin{claim}~\label{claim}
Let $G$ be a planar graph. For any subset $U\subseteq V(G)$ with
$\card{\bd_G U}\le K\sqrt{\card U},$ there is an ordering of $U$ such that
every prefix $P$ satisfies  $\card{\bd_G P}\le C\sqrt{\card P}.$
\end{claim}

Note that $\bd_G P$ may contain many vertices outside $U$.

\begin{proof} By induction on $m=|U|$.

Suppose first that $1\le m\le M$. Apply Lemma~\ref{lem:uniform} to $G[U]$.
For every prefix $P$ in the resulting ordering, we have
\[
  \bd_G P\subseteq\bd_G U\cup\bd_{G[U]}P.
\]
Thus, for every nonempty $P$,
\[
  \card{\bd_G P}
  \le(K+D)\sqrt m
  \le(K+D)\sqrt M
  \le C\sqrt{\card P}.
\]
This establishes the base case.

Now let $m>M$. Choose $W$ by Lemma~\ref{lem:shrink}, and set $S=U\setminus W$.
Since $\card W<m$ and, by \eqref{eq:shrink}, $|\bd_G W|\le K\sqrt{|W|}$, we can use the induction hypothesis to conclude that $W$ has an ordering that meets the requirements of Claim~\ref{claim}. Append to it an
ordering of $S$, obtained by applying Lemma~\ref{lem:uniform} to the graph $G[S]$.

A nonempty prefix $P$ entirely contained in $W$ satisfies $\card{\bd_G P}\le C\sqrt{\card P},$ by
induction. Any remaining prefix has the form $P=W\cup Q$, where $Q$ is a
prefix of the ordering of $S$. Then we have
\begin{equation}\label{eq:boundary-union}
  \bd_G P
  \subseteq
  \bd_G U\;\cup\;(\bd_G W\setminus P)\;\cup\;\bd_{G[S]}Q.
\end{equation}
Indeed, a neighbour of $P$ outside $P$ either lies outside $U$, or lies in
$S\setminus Q$ and is adjacent to a vertex of $W$ or of $Q$, respectively.
Therefore, we have
\begin{align*}
  \card{\bd_G P}
  &\le K\sqrt m+K\sqrt{\card W}+D\sqrt{\card S}\\
  &\le(2K+D)\sqrt m\\
  &\le(2K+D)\sqrt{10/3}\sqrt{\card P}
   \le C\sqrt{\card P},
\end{align*}
using $\card P\ge\card W\ge3m/10$. This proves the claim.
\end{proof}

In the special case $U=V(G)$, we have $|\bd_G U|=0$, and Claim~\ref{claim} gives the upper bound of
\Cref{thm:main}.

The second part of \Cref{thm:main} follows by taking $G$ to be the $\sqrt{n} \times \sqrt{n}$ grid. For simplicity we assume
that  $\sqrt{n}$ is an integer; a similar argument can be made to work for the general case. It remains to show that any $S \subset V(G)$, $|S| \leq n/2$, has $\card{ \bd_G S } = \Omega \left( \sqrt{|S|} \right)$.

Let $S$ be a subset of the $\sqrt{n} \times \sqrt{n}$ grid, with $i = |S| \le n/2$.
Let $t$ be the maximum number of points of $S$ in any row.
Note that since $i \le n/2$, we have $\frac{i}{\sqrt n} \le \frac{\sqrt i}{\sqrt 2}$, and hence
$$\sqrt{i} - \frac{i}{\sqrt{n}} \ge \left(1-\frac{1}{\sqrt2}\right)\sqrt{i} = \Omega(\sqrt{i}).$$
Now we prove that the number of boundary edges
is $\Omega \left( \sqrt{i} \right)$, 
via these  two cases.
\begin{description}
	\item[$t \geq \sqrt{i}$.] Let $R \subseteq S$ be the points of $S$ in a row with $|R| \geq \sqrt{i}$, and consider the $|R|$ columns containing the points of $R$. The number of these columns in which all $\sqrt{n}$ points belong to $S$ is at most $i/\sqrt{n}$. Thus at least
	$$|R| - \frac{i}{\sqrt{n}} \geq \sqrt{i}  - \frac{i}{\sqrt{n}} = \Omega \left( \sqrt{i} \right) $$
	such columns are non-full, and thus at least one vertical boundary edge lies in each of these columns, giving the required bound.
	\item[$t < \sqrt{i}$.] Then there are at least $\sqrt{i}$ rows of the grid containing some point of $S$. Out of these rows, at least
	$$\sqrt{i} - \frac{i}{\sqrt{n}} = \Omega \left( \sqrt{i} \right) $$
	are non-full, and so each of these contains at least one horizontal boundary edge.
\end{description}
Since each boundary edge is counted at its endpoint outside $S$, and each boundary vertex has degree at most $4$, we get  $|\bd_G S|=\Omega(\sqrt{i})$.

\section{Concluding remarks}

There are several useful separator theorems for hereditary families of graphs. A family $\F$ of graphs is called \emph{hereditary} if for every graph $G\in\F$, every induced subgraph of $G$ also belongs to $\F$. The family of planar graphs is obviously hereditary. So is the family of \emph{string graphs}, that is, the family of graphs whose vertices can be represented by continuous arcs (``strings'') in the plane such that two curves have nonempty intersection if and only if the corresponding vertices are connected by an edge. 

Improving some weaker results \cite{FP10, FP14, Ma14}, Lee~\cite{Lee17} published a proof of a conjecture of Fox and Pach~\cite{FP10}, according to which every string graphs with $m$ edges admits a balanced vertex separator of size $O(\sqrt{m})$. Recently, a flaw has been discovered in Lee's argument, and a correct proof has been supplied by Davies \emph{et al.}~\cite{DRSW26}. Using their result as a black box, one can prove the following statement analogous to Theorem~\ref{thm:main}.

\begin{theorem}\label{thm:string}
There is an absolute constant $C$ such that every string graph $G$ has
a vertex ordering $v_1,\ldots,v_n$ with the following property. Let
$P_i=\{v_1,\ldots,v_i\}$ and let $m_i$ denote the number of edges in the subgraph of $G$ induced by $P_i$.  

Then, for every $1\le i\le n$,
the prefix $P_i$ can be separated from $V(G)\setminus P_i$ by
deleting at most $C\sqrt{m_i+1}$ vertices belonging to $P_i$.
\end{theorem}

Writing $\sqrt{m_i+1}$ instead of $\sqrt{m_i}$ is merely a technicality, to handle the trivial cases where $m_i=0$.

Note that every planar graph is string graph, and the number of edges induced by an $i$-element subset of the vertex set of a planar graph is $O(i)$. In particular, in an ordering of the vertex set of a planar graphs, guaranteed by Theorem~\ref{thm:string}, every prefix $P_i$ can be separated from its complement by deleting $O(\sqrt{i})$ vertices of $G$. This is almost the same as Theorem~\ref{thm:main}.

It is tempting to believe that, as in Theorem~\ref{thm:main}, here we can also guarantee the existence of an ordering such that 
\[
  \card{\bd_G \{v_1,\ldots,v_i\}}\le O(\sqrt{m_i})\qquad(\mbox{for every } 1\le i\le n).
\]
However, this is not the case. To see this, notice that for the complete graph $K_n$, we have $m_i={i\choose 2}$ and $\sqrt{m_i}<i$. If we are not allowed to delete any vertex of a prefix $P_i$, then to separate $P_i$ from its complement, we need to delete all $n-i$ remaining vertices.
\medskip

\noindent\textbf{Acknowledgement.} This note was written by humans for humans while the second author was visiting Universit\`e Sorbonne Paris Nord in September 2026. The details of the proof of Theorem~\ref{thm:string}, omitted here, were worked out by ChatGPT.

\end{document}